\documentclass[a4paper,12pt]{article}
\usepackage{amsmath,amsthm,amssymb,graphicx,multicol}
\usepackage{fullpage,hyperref,datetime,enumitem}

\let\originalleft\left
\let\originalright\right
\renewcommand{\left}{\mathopen{}\mathclose\bgroup\originalleft}
\renewcommand{\right}{\aftergroup\egroup\originalright}

\numberwithin{equation}{section}

\theoremstyle{plain}
\newtheorem{lemma}{Lemma}[section]

\theoremstyle{plain}

\theoremstyle{plain}
\newtheorem{theorem}[lemma]{Theorem}

\theoremstyle{plain}
\newtheorem{corollary}[lemma]{Corollary}

\theoremstyle{definition}
\newtheorem{definition}[lemma]{Definition}

\theoremstyle{remark}
\newtheorem{remark}[lemma]{Remark}

\newcommand{\bR}{{\mathbb{R}}}
\newcommand{\bN}{{\mathbb{N}}}
\newcommand{\cP}{{\mathcal{P}}}
\newcommand{\cQ}{{\mathcal{Q}}}
\newcommand{\cR}{{\mathcal{R}}}

\newcommand{\SMALL}{\textstyle}

\newcommand{\logg}{\operatorname{logg}}

\newcommand{\iphi}{{\Phi_\alpha}}  

\newcommand{\const}{{\mathbf{C}}}

\begin{document}

\title{Linear response for intermittent maps with summable and nonsummable
decay of correlations}

\author{Alexey Korepanov
\thanks{Mathematics Institute, University of Warwick, Coventry, CV4 7AL, UK}} 
\date{August 26, 2015 \\ updated March 15, 2016}
\maketitle

\begin{abstract}
  We consider a family of Pomeau-Manneville type interval maps $T_\alpha$,
  parametrized by $\alpha \in (0,1)$, with the unique absolutely continuous invariant 
  probability measures $\nu_\alpha$, and rate of correlations decay $n^{1-1/\alpha}$. 
  We show that despite the absence of a spectral gap for all $\alpha \in (0,1)$ and
  despite nonsummable correlations for $\alpha \geq 1/2$, the map 
  $\alpha \mapsto \int \varphi \, d\nu_\alpha$ is continuously differentiable 
  for $\varphi \in L^{q}[0,1]$ for $q$ sufficiently large.
\end{abstract}

\section{Introduction}

Let $T_\alpha \colon X \to X$ be a family of transformations on a Riemannian manifold $X$
parametrized by $\alpha$ and admitting unique SRB measures $\nu_\alpha$.
Having an observable $\varphi\colon X \to \bR$, it may be important to know how
$\int \varphi \, d\nu_\alpha$ changes with $\alpha$. If the map
$\alpha \mapsto \int \varphi \, d\nu_\alpha$ is differentiable, then
\emph{linear response} holds.

An interesting question is, which families of maps and observables have linear response.
Ruelle proved linear response in the Axiom A case \cite{R97,R98, R09,R09.1}.
It was shown in \cite{D04,B07,M07,BS08} that spectral gap and structural stability are not
necessary or sufficient conditions.

We consider a family of Pomeau-Manneville type maps with slow (polynomial) decay of correlations:
$T_\alpha \colon [0,1] \to [0,1]$, given by
\begin{equation}
  \label{eq:LSVT}
  T_\alpha (x) = \begin{cases} 
    x(1+2^\alpha x^\alpha) & \text{ if } \,\, x \in [0,1/2] \\ 
    2x-1 & \text{ if } \,\, x \in (1/2,1] 
  \end{cases},
\end{equation}
parametrized by $\alpha \in [0,1)$. By \cite{LSV99}, each $T_\alpha$ admits a unique absolutely
continuous invariant probability measure $\nu_\alpha$, and the sharp rate of decay of 
correlations for H\"older observables is $n^{1-1/\alpha}$ \cite{Y99,S02,G04,H04}.


We prove linear response on the interval $\alpha \in (0,1)$, including the case
when $\alpha \geq 1/2$, and correlations are not summable.
This is the first time that linear response has been proved in the case 
of nonsummable decay of correlations.
We develop a machinery which, when applied to the family $T_\alpha$, yields:
\begin{theorem}
  \label{th:LSVmain}
  For any $\varphi \in C^1[0,1]$, the map $\alpha \mapsto \int \varphi \, d\nu_\alpha$
  is continuously differentiable on $(0,1)$.
\end{theorem}

Using additional structure of the family $T_\alpha$, we prove a stronger result:

\begin{theorem}
  \label{th:whatamIdoing}
  Let $\rho_\alpha$ be the density of $\nu_\alpha$.
  For every $\alpha, x \in (0,1) \times (0,1]$ there exists
  a partial derivative $\partial_\alpha \rho_\alpha (x)$.
  Both $\rho_\alpha(x)$ and $\partial_\alpha \rho_\alpha(x)$ are jointly continuous in $\alpha,x$
  on $(0,1) \times (0,1]$.
  Moreover, for every interval $[\alpha_-, \alpha_+] \subset (0,1)$
  there exists a constant $K$, such that for all $x \in (0,1]$ and $\alpha \in
  [\alpha_-, \alpha_+]$
  \[
    \rho_\alpha (x) \leq K \, x^{-\alpha}
    \qquad \text{and} \qquad
    |\partial_\alpha \rho_\alpha(x)| \leq K \, x^{-\alpha} \left(1-\log x\right).
  \]
  In particular, for any $q > (1-\alpha_+)^{-1}$ and 
  observable $\varphi \in L^{q}[0,1]$, the
  map $\alpha \mapsto \int \varphi \, d\nu_\alpha$ is
  continuously differentiable on $[\alpha_-, \alpha_+]$.
\end{theorem}

The same problem of linear response for the family $T_\alpha$ has been solved
independently by Baladi and Todd \cite{BT15} using different methods.
They prove that for \(\alpha_+ \in (0,1)\),
$q > (1-\alpha_+)^{-1}$ and $\varphi \in L^{q}[0,1]$, the
map $\alpha \mapsto \int \varphi \, d\nu_\alpha$ is
differentiable on $[0, \alpha_+)$,
plus they give an explicit formula for the derivative
in terms of the transfer operator corresponding to \(T_\alpha\).
We obtain more control of the invariant measure, as in
Theorem~\ref{th:whatamIdoing}, but do not give such a formula.
Instead we provide explicit formulas for \(\rho_\alpha\) and 
\(\partial_\alpha \rho_\alpha\) in terms of the transfer operator 
for the induced map (see 
Subsections~\ref{subsection:scheme} and 
\ref{subsection:proofofth:whatamIdoing}),
but we do not state them here because they are too technical.
Whereas [BT15] were the first to treat the case $\alpha<1/2$, 
we were the first to treat the case $\alpha\ge 1/2$.

In a more recent paper \cite{BS15}, Bahsoun and Saussol consider
a class of dynamical systems which includes \eqref{eq:LSVT}.
They prove in particular that for \(\beta \in (0,1)\) and 
\(\alpha \in (0,\beta)\),
\[
  \lim_{\varepsilon \to 0} \sup_{x \in (0,1]} x^\beta \Bigl|
    \frac{\rho_{\alpha+\varepsilon} - \rho_\alpha}{\varepsilon}
    - \partial_\alpha \rho_\alpha
  \Bigr| = 0.
\]
That is, \(\rho_\alpha\) is differentiable as an element of
a Banach space of continuous functions on \((0,1]\) with a norm
\(\|\varphi\| = \sup_{x \in (0,1]} x^\beta |\varphi(x)|  \).
We remark that for the map \(\eqref{eq:LSVT}\) this follows
from Theorem~\ref{th:whatamIdoing}.

The paper is organized as follows. In Section \ref{section:setup} we introduce
an abstract framework, and in Section \ref{section:LSVintro} we 
apply it to the family $T_\alpha$, and prove Theorem \ref{th:LSVmain}.

Technical parts of proofs are presented separately: in Section \ref{section:dhdalpha}
for the abstract framework, and in Section \ref{section:LSV}
for the properties of the family $T_\alpha$.

Theorem \ref{th:whatamIdoing} 
is proven in Subsection 
\ref{subsection:proofofth:whatamIdoing}. We do not give the proof earlier in the
paper, because it uses rather special technical properties of $T_\alpha$.

\section{Setup and Notations}
\label{section:setup}

Let $I \subset \bR$ be a closed bounded interval, and 
$F_\alpha: I \to I$ be a family of maps,
parametrized by $\alpha \in [\alpha_-, \alpha_+]$.
Assume that each $F_\alpha$ has finitely or countably many full branches,
indexed by $r \in \cR$, the same set $\cR$ for all $\alpha$.

Technically we assume that $I = \bigcup_r [a_r,b_r]$ modulo a zero measure set
(branch boundaries $a_r$ and $b_r$ may depend on $\alpha$), and that for each r
the map $F_{\alpha,r}:[a_r,b_r] \to I$ is a diffeomorphism;
here $F_{\alpha,r}$ equals to $F_\alpha$ on $(a_r, b_r)$, and is extended countinuously
to $[a_r, b_r]$.

\begin{itemize}
  \item We use the letter $\xi$ for spatial variable,
    and notation $(\cdot)'$ for differentiation with respect to $\xi$, and
    $\partial_\alpha$ for differentiation with respect to $\alpha$.
  \item Denote $G_{\alpha,r} = \left|(F_{\alpha,r}^{-1})'\right|$, defined on $I$.
    Note that $G_{\alpha,r} = \pm (F_{\alpha,r}^{-1})'$, the sign depends only on $r$.
  \item For each $i$ let 
    $\|h\|_{C^i} = \max(\|h\|_\infty, \|h'\|_\infty, \ldots, \|h^{(i)}\|_\infty)$ 
    denote the $C^i$ norm of $h$.
  \item Let $m$ be the normalized Lebesgue measure on $I$, and $P_\alpha$ be the transfer operator 
    for $F_\alpha$ with respect to $m$. By definition,
    $\int (P_\alpha u) v \, dm = \int u (v \circ F_\alpha) \, dm$ for $u \in L^1 (I)$ and
    $v \in L^\infty (I)$. There is an explicit formula for $P_\alpha$:
    \[
      (P_\alpha h)(\xi) = \sum_r G_{\alpha,r}(\xi) h(F_{\alpha,r}^{-1}(\xi)).
    \]
\end{itemize}

We assume that $F_{\alpha,r}^{-1}$ and $G_{\alpha,r}$, as functions of 
$\alpha$ and $\xi$, have continuous second order partial derivatives
for each $r \in \cR$, and there are constants 
\[
  0 < \sigma < 1, \qquad K_0 > 0
  \qquad \text{and} \qquad \gamma_r \geq 1, \, r \in \cR
\]
such that
uniformly in $\alpha \in [\alpha_-, \alpha_+]$ and $r \in \cR$:
\begin{samepage}
\begin{multicols}{2}
\begin{enumerate}[label=A\arabic*., ref=A\arabic*]
  \item\label{A1} 
    $\|G_{\alpha,r}\|_\infty \leq \sigma$,
  \item\label{A2} 
    $\|G_{\alpha,r}'  \big/ G_{\alpha,r}\|_\infty  \leq K_0$,
  \item\label{A3}
    $\|G_{\alpha,r}'' \big/ G_{\alpha,r}\|_\infty  \leq K_0$,
  \columnbreak
  \item\label{A4} 
    $\left\|\partial_\alpha F_{\alpha,r}^{-1}\right\|_\infty \leq \gamma_r $,
  \item\label{A5} 
    $\left\|(\partial_\alpha G_{\alpha,r})\big/G_{\alpha,r} \right\|_\infty \leq \gamma_r $,
  \item\label{A6}
    $\left\|(\partial_\alpha G_{\alpha,r}')\big/G_{\alpha,r}\right\|_\infty \leq \gamma_r $,
  \item\label{A7} 
    $\sum_r \left\|G_{\alpha,r}\right\|_\infty \gamma_r  \leq K_0$.
\end{enumerate}
\end{multicols}
\end{samepage}

It is well known that under conditions \ref{A1}, \ref{A2}, the map 
$F_\alpha$ admits a unique absolutely continuous invariant
measure (see for example \cite{P80,Z04}), which we denote by 
$\mu_\alpha$, and its density by $h_\alpha = d\mu_\alpha / dm$.

\begin{theorem}
  \label{th:dhdalpha}
  $h_\alpha \in C^2(I)$ and $\partial_\alpha h_\alpha \in C^1(I)$
  for each $\alpha \in [\alpha_-, \alpha_+]$.
  The maps
  \[
    \begin{alignedat}{1}
      [\alpha_-, \alpha_+] & \longrightarrow C^2(I)\\
      \alpha               & \longmapsto h_\alpha 
    \end{alignedat}
    \qquad \text{and} \qquad 
    \begin{alignedat}{1}
      [\alpha_-, \alpha_+] & \longrightarrow C^1(I)\\
      \alpha               & \longmapsto \partial_\alpha h_\alpha 
    \end{alignedat}
  \]
  are continuous.  
\end{theorem}

The proof of Theorem \ref{th:dhdalpha} is postponed until 
Section \ref{section:dhdalpha}.

\begin{remark}
  Later in the proof of Theorem \ref{th:dhdalpha}
  we explicitly compute constants $K_{1}$ and $K_{2}$, 
  such that $\|h_\alpha\|_{C^2} \leq K_{1}$ and
  $\|\partial_\alpha h_\alpha\|_{C^1} \leq K_{2}$.
  These constants depend only on $K_{0}$ and $\sigma$.
  Below we use $K_{1}$ and $K_{2}$ as reference bounds on
  $\|h_\alpha\|_{C^2}$ and
  $\|\partial_\alpha h_\alpha\|_{C^1}$.
\end{remark}

\begin{remark}
  Both $h_\alpha(\xi)$ and $(\partial_\alpha h_\alpha)(\xi)$ are continuous
  in $\alpha$, and continuous in $\xi$ uniformly in $\alpha$, 
  because $\|h_\alpha\|_{C^2} \leq K_{1}$.
  Therefore both are jointly continuous in $\alpha$ and $\xi$.
\end{remark}

\begin{corollary}
  \label{cor:linresp}
  Assume that $\Phi_\alpha$ is a family of observables,  such that
  $\Phi_\alpha (F_{\alpha,r}^{-1}(\xi))$ and
  $\partial_\alpha [\Phi_\alpha (F_{\alpha,r}^{-1}(\xi))]$
  are jointly continuous in $\alpha$ and $\xi$ for each $r$,
  and
  \[
    \|\Phi_\alpha \circ F_{\alpha,r}^{-1}\|_\infty \leq \delta_r,
    \qquad
    \|\partial_\alpha [\Phi_\alpha \circ F_{\alpha,r}^{-1}] \|_\infty \leq \delta_r
  \]
  for some constants $\delta_r\geq1$, $r \in \cR$.
  Assume also that 
  $
    \sum_r \gamma_r \, \delta_r \,\left\| G_{\alpha,r} \right\|_\infty \leq K_{3}.
  $
  Then the map $\alpha \mapsto \int \Phi_\alpha\, d\mu_\alpha$ is 
  continuously differentiable on $[\alpha_-, \alpha_+]$.
\end{corollary}
\begin{proof}
  First, 
  \begin{align*}
    \int \Phi_\alpha \, d \mu_\alpha 
    & = \int h_\alpha \Phi_\alpha \, dm 
    = \int P_\alpha(h_\alpha \Phi_\alpha) \, dm \\
    & = \int \left( \sum_r (h_\alpha \circ F_{\alpha,r}^{-1}) 
      (\Phi_\alpha\circ F_{\alpha,r}^{-1}) G_{\alpha,r} \right) \, dm.
  \end{align*}
  By Theorem \ref{th:dhdalpha} and our assumptions,
  $h_\alpha \circ F_{\alpha,r}^{-1}$, 
  $\Phi_\alpha\circ F_{\alpha,r}^{-1}$ and $G_{\alpha,r}$
  are jointly continuous in $\alpha$ and $\xi$, and
  have jointly continuous partial derivatives by $\alpha$. 
  Since $\|h_\alpha \|_{C^2} \leq K_{1}$,
  $\|\Phi_\alpha \circ F_{\alpha,r}^{-1}\|_\infty \leq \delta_r$
  and $\sum_r \delta_r \, \gamma_r \, \|G_{\alpha,r}\|_{\infty} \leq K_{3}$,
  the series inside the integral converges uniformly to a function
  which is jointly continuous in $\alpha$ and $\xi$. Therefore,
  $\int \Phi_\alpha \, d \mu_\alpha$ depends continuously on $\alpha$.
  
  Moreover, since
  $\|\partial_\alpha h_\alpha\|_{C^1} \leq K_{2}$,
  $\|\partial_\alpha F_{\alpha,r}^{-1}\|_\infty \leq \gamma_r$,
  $|(\Phi_\alpha\circ F_{\alpha,r}^{-1})| \leq \delta_r$,
  $|\partial_\alpha (\Phi_\alpha\circ F_{\alpha,r}^{-1})|\leq \delta_r$
  and
  $|\partial_\alpha [G_{\alpha,r}]| \leq \gamma_r \, G_{\alpha,r}$,
  we can write
  \begin{align*}
    \left|\partial_\alpha (h_\alpha \circ F_{\alpha,r}^{-1})\right|
    & = \left| [\partial_\alpha h_\alpha]\circ F_{\alpha,r}^{-1}
      + (h_\alpha'\circ F_{\alpha,r}^{-1}) \, \partial_\alpha F_{\alpha,r}^{-1} \right|\\
    & \leq K_{2} + K_{1} \gamma_r \qquad \text{and} \\
    \left|
      \partial_\alpha 
      \left[ (h_\alpha \circ F_{\alpha,r}^{-1}) 
      (\Phi_\alpha\circ F_{\alpha,r}^{-1}) G_{\alpha,r} \right]
    \right|
    & \leq \left[(K_{2}+K_{1} \gamma_r) + K_{1} + K_{1} \gamma_r
    \right] \, \delta_r \, G_{\alpha,r} \\
    & \leq \left(3 K_{1} + K_{2}\right) \, \delta_r \, \gamma_r \, G_{\alpha,r}.
  \end{align*}
  Since $\sum_r \delta_r \, \gamma_r \, \|G_{\alpha,r}\|_{\infty} \leq K_{3}$, we can write
  \begin{align*}
    \frac{d}{d\alpha} \int \Phi_\alpha \, d \mu_\alpha 
    & = \frac{d}{d\alpha}  \int \left( \sum_r (h_\alpha \circ F_{\alpha,r}^{-1}) 
      (\Phi_\alpha\circ F_{\alpha,r}^{-1}) G_{\alpha,r} \right) \, dm \\
    & = \sum_r \int \partial_\alpha 
      \left[ (h_\alpha \circ F_{\alpha,r}^{-1}) 
      (\Phi_\alpha\circ F_{\alpha,r}^{-1}) G_{\alpha,r} \right] \, dm.     
  \end{align*}
  The series converges uniformly, and is bounded by $K_{3} \left(3 K_{1} + K_{2}\right)$.
  The terms are continuous in $\alpha$, thus so is the sum.
\end{proof}

\section{Application to Pomeau-Manneville type maps}
\label{section:LSVintro}

In this section we work with the family of maps $T_\alpha$, defined
by equation (\ref{eq:LSVT}).
Assume that $\alpha \in [\alpha_-,\alpha_+] \subset (0,1)$.
Let $\tau_\alpha(x) = \min \{ k\geq 1 \colon T_\alpha^k x \in [1/2,1] \}$ 
be the return time to the interval $[1/2,1]$.
Let
\[
  F_\alpha \colon [1/2,1] \to [1/2,1], \qquad \qquad
        x \mapsto T_\alpha^{\tau_\alpha(x)}(x)
\]
be the induced map. 

\paragraph{Branches.} Let $x_0 = 1$, $x_1=1/2$, and define 
$x_k \in (0,1/2]$ for $k \geq 1$ by setting 
$x_k = T_\alpha x_{k+1}$.
Note that $T_\alpha^k \colon (x_{k+1}, x_k) \to (1/2,1)$
is a diffeomorphism. Let $y_k = (1+x_k)/2$. Then 
$T_\alpha^{k+1} \colon (y_{k+1}, y_k) \to (1/2,1)$
is a diffeomorphism.
It is clear that $\tau_\alpha = k+1$ on $(y_{k+1}, y_k)$,
so the map $F_\alpha$ has full branches on the 
intervals $(y_{k+1},y_k)$ for $k\geq 0$.

We index branches by $r \in \cR = \bN \cup \{0\}$, the $r$-th branch being 
the one on $(y_{r+1},y_r)$.
Let $F_{\alpha,r}: [y_{r+1}, y_r] \to [0,1]$ be the continuous extension of
$F_\alpha: (y_{r+1}, y_r) \to (0,1)$.

For notational convenience we introduce a function
\[
  \logg(r) = \begin{cases} 1 & r \leq e \\ \log(r)  & r > e \end{cases}.
\]

Let $\varphi \in C^1[0,1]$ be an observable;
let 
\begin{equation}
  \label{eq:anr3}
  \iphi = \sum_{k=0}^{\tau_\alpha-1} \varphi \circ  T_\alpha^k
\end{equation}
be the corresponding observable for the induced system. 

\begin{theorem}
  \label{th:lsv}
  The family of maps $F_\alpha = T_\alpha^{\tau_\alpha} \colon [1/2,1] \to [1/2,1]$ with
  observables $\iphi$ fits into the setup of Theorem \ref{th:dhdalpha}
  and Corollary \ref{cor:linresp}
  with branches indexed as above, $\delta_r = K (r+1) \, \|\varphi\|_{C^1}$, 
  $\sigma = 1/2$, and $\gamma_r = K (\logg r)^3$,
  where $K$ is a constant, depending only on $\alpha_-$ and $\alpha_+$.
\end{theorem}

The proof consists of verification of the assumptions of
Theorem~\ref{th:dhdalpha} and Corollary~\ref{cor:linresp},
and is carried out in Subsection \ref{subsection:proofofth:lsv}.
Here we use Theorem \ref{th:lsv} to prove our main result~--- 
Theorem~\ref{th:LSVmain}.

\begin{proof}[Proof of Theorem \ref{th:LSVmain}]
  The invariant measure $\nu_\alpha$ for $T_\alpha$ is related to the
  invariant measure $\mu_\alpha$ for $F_\alpha$ by Kac's formula:
  \[
    \int \varphi \, d\nu_\alpha = \int \iphi \, d \mu_\alpha \; \Big/
    \int \tau_\alpha d \mu_\alpha,
  \]
  where \(\iphi\) is given by (\ref{eq:anr3}).
  Note that if $\varphi \equiv 1$, then $\iphi = \tau_\alpha$.
  By Theorem \ref{th:lsv}, both integrals are continuously
  differentiable in $\alpha$. Also, $\tau_\alpha \geq 1$, so
  $\int \tau_\alpha d \mu_\alpha \geq 1$.
  Hence $\int \varphi \, d\nu_\alpha$ is continuously differentiable
  in $\alpha$.
\end{proof}

\section{Proof of Theorem \ref{th:dhdalpha}}
  \label{section:dhdalpha}

For $h \in C^1(I)$ and $\alpha \in [\alpha_-, \alpha_+]$ define
$Q_\alpha h = \partial_\alpha (P_\alpha h)$, if the derivative
exists. Denote
\[
  \cQ_h (\alpha) = Q_\alpha h,
  \qquad \qquad
  \cP_h (\alpha) = P_\alpha h.
\]

\subsection{Outline of the proof}
  \label{subsection:scheme}

The proof consists of three steps:
\paragraph{(a) Continuity (Subsection \ref{subsection:continuity}).}
We show that for $i=1,2$ the linear operators
\[
  P_\alpha \colon C^i(I) \to C^i(I) \qquad \text{and} \qquad
  Q_\alpha \colon C^i(I) \to C^{i-1}(I)
\]
are well defined, and their norms are bounded uniformly in $\alpha$.
Plus, they continuously depend on $\alpha$ in the following sense:
for each $h\in C^i(I)$ the maps
$\cP_h \colon [\alpha_-, \alpha_+] \to C^i(I)$
and
$\cQ_h \colon [\alpha_-, \alpha_+] \to C^{i-1}(I)$
are continuous.
Moreover, the map $\cP_h \colon [\alpha_-, \alpha_+] \to C^{i-1}(I)$
is continuously differentiable, and its derivative
is $\cQ_h \colon [\alpha_-, \alpha_+] \to C^{i-1}(I)$.

In addition, $\int Q_\alpha h \, dm = 0$ for every $h \in C^1(I)$.

\paragraph{(b) Distortion bounds and coupling
(Subsection \ref{subsection:distortionsandcoupling}).}
\begin{itemize}
  \item If $h$ is in $C^i(I)$ for $i=1 \text{ or } 2$, and 
    $\int h \, dm = 0$, then $\|P_\alpha^k h\|_{C^i} \to  0$ exponentially fast,
    uniformly in $\alpha$.
  \item If $h \in C^2(I)$ and $\int h \, dm = 1$, then
    $
      h_\alpha 
      = \lim_{k\to\infty} P_\alpha^k h 
      = h + \sum_{k=0}^\infty P_\alpha^k (P_\alpha h - h)
    $
  \item The series above converges exponentially fast in $C^2$, and 
    $\|h_\alpha\|_{C^2}$ is bounded on $[\alpha_-, \alpha_+]$.
  \item The map $\alpha \mapsto h_\alpha$ from $[\alpha_-, \alpha_+]$ to $C^2(I)$
    is continuous.
\end{itemize}

\paragraph{(c) Computation of $\partial_\alpha h_\alpha$.}
Fix $\alpha \in [\alpha_-, \alpha_+]$.
Start with a formula, which holds for every $\alpha$, $\beta$ and $n$:
\begin{equation*}
  P_\beta^n h_\alpha - P_\alpha^n h_\alpha
  = \sum_{k=0}^{n-1} P_\beta^k (P_\beta - P_\alpha) h_\alpha.
\end{equation*}
Since the terms in the above sum converge exponentially fast in $C^2$,
we can take the limit $n \to \infty$. Note that $P_\alpha^n h_\alpha = h_\alpha$
and $\lim_{n \to \infty} P_\beta^n h_\alpha = h_\beta$. Hence
\begin{equation*}
  h_\beta - h_\alpha
  = \sum_{k=0}^{\infty} P_\beta^k (P_\beta - P_\alpha) h_\alpha.
\end{equation*}
For fixed $\alpha$, recall that the map 
$\beta \mapsto P_\beta h_\alpha$ from
$[\alpha_-, \alpha_+]$ to $C^1(I)$ is continuously differentiable, 
and its derivative is the map $\beta \mapsto Q_\beta h_\alpha$,
hence
\[
  (P_\beta - P_\alpha) h_\alpha
  = (\beta - \alpha) \; Q_\alpha h_\alpha + R_\beta,
\]
with $\|R_\beta\|_{C^1} = o(\beta - \alpha)$. 
Note that both $Q_\alpha h$ and $R_\beta$ have zero mean. Next,
\[
  h_\beta - h_\alpha
  = (\beta - \alpha) \sum_{k=0}^{\infty} P_\beta^k Q_\alpha h_\alpha
  + \sum_{k=0}^{\infty} P_\beta^k R_\beta.
\]
Both series converge exponentially fast in $C^1(I)$, uniformly in $\alpha$ and $\beta$,
and the $C^1$ norm of the second one is $o(\alpha - \beta)$.

Observe that the maps 
  \[
    \begin{alignedat}{1}
      [\alpha_-, \alpha_+] \times C^1(I) & \longrightarrow C^1(I)\\
      \alpha, h               & \longmapsto P_\alpha h
    \end{alignedat}
    \qquad \text{and} \qquad 
    \begin{alignedat}{1}
      [\alpha_-, \alpha_+] \times C^2(I) & \longrightarrow C^1(I)\\
      \alpha,h               & \longmapsto Q_\alpha h 
    \end{alignedat}
  \]
are continuous in $\alpha$, and continuous in $h$ uniformly in $\alpha$,
because $P_\alpha$ and $Q_\alpha$ are linear operators, 
bounded uniformly in $\alpha$.
Thus both maps are jointly continuous in $\alpha$ and $h$.
Recall that the map $\alpha \to h_\alpha$ from $[\alpha_-, \alpha_+]$
to $C^2(I)$ is continuous. Thus the map
$\alpha, \beta \mapsto P_\beta^k Q_\alpha h_\alpha$
from $[\alpha_-, \alpha_+]^2$ to $C^1(I)$ is continuous.

Therefore, in $C^1(I)$ topology,
\[
  \lim_{\beta \to \alpha}
  \frac{h_\beta - h_\alpha}{\beta - \alpha} = \sum_{k=0}^{\infty} P_\alpha^k Q_\alpha h_\alpha,
\]
and $\sum_{k=0}^{\infty} P_\alpha^k Q_\alpha h_\alpha$ continuously
depends on $\alpha$.
Note that the above also implies that
$\partial_\alpha h_\alpha = \sum_{k=0}^\infty P_\alpha^k Q_\alpha h_\alpha$
(if understood pointwise).

Therefore the map $\alpha \mapsto h_\alpha$ from $[\alpha_-, \alpha_+]$ to
$C^1(I)$ is continuously differentiable, and its derivative is
$\alpha \mapsto \partial_\alpha h_\alpha = \sum_{k=0}^{\infty} P_\alpha^k Q_\alpha h_\alpha$.

In the remainder of this section we make the above precise.

\subsection{Continuity}
\label{subsection:continuity}

\begin{lemma}
  \label{lemma:diffalpha1}
  Let $K_{6} = 4 K_0 (1 + K_0)$.
  For each $i=1,2$ and $h \in C^i(I)$:
  \begin{enumerate}[label=\alph*), ref=\alph*]
    \item
      The map $\cP_h \colon [\alpha_-, \alpha_+] \to C^i(I)$ is continuous.
      Also, $\|P_\alpha h \|_{C^i} \leq K_{6} \|h \|_{C^i}$.
    \item 
      The map $\cQ_h \colon [\alpha_-, \alpha_+] \to C^{i-1}(I)$ is continuous.
      Also, $\|Q_\alpha h\|_{C^{i-1}} \leq K_{6} \|h \|_{C^i}$.
    \item
      The map $\cQ_h \colon [\alpha_-, \alpha_+] \to C^{i-1}(I)$
      is the derivative of the map $\cP_h \colon [\alpha_-, \alpha_+] \to C^{i-1}(I)$.
    \item 
      $\int Q_\alpha h\, dm = 0$.
  \end{enumerate}

\end{lemma}
\begin{proof}
  We do the case $i=2$; the case $i=1$ is similar and simpler.
  \begin{enumerate}[label=\alph*), ref=\alph*]
    \item
      Let $p_{\alpha,r} = (h \circ F_{\alpha,r}^{-1}) G_{\alpha,r}$. Then
      \begin{align*}
        p_{\alpha,r}'
        & = \pm (h'\circ F_{\alpha,r}^{-1}) G_{\alpha,r}^2 
          + (h\circ F_{\alpha,r}^{-1}) \, G'_{\alpha,r}, 
          \quad \text{and}\\
        p_{\alpha,r}''
        & = (h''\circ F_{\alpha,r}^{-1}) G_{\alpha,r}^3
            \pm 3 (h'\circ F_{\alpha,r}^{-1}) G_{\alpha,r} G_{\alpha,r}'
            + (h\circ F_{\alpha,r}^{-1}) G_{\alpha,r}'',
      \end{align*}
      where the sign of $\pm$ depends only on $r$. 
      By assumptions \ref{A1}, \ref{A2} and \ref{A3},
      $\|p_{\alpha,r}\|_{C^2} \leq (4 K_0 + 1)  \|h\|_{C^2} \|G_{\alpha,r}\|_{\infty}$.
      
      Since $p_{\alpha,r}$, $p_{\alpha,r}'$ and $p_{\alpha,r}''$
      are jointly continuous in $\alpha$ and $\xi$,
      we obtain that the map $\alpha \mapsto p_{\alpha,r}$ from $[\alpha_-, \alpha_+]$
      to $C^2(I)$ is continuous.
      
      By assumption \ref{A7}, $\sum_r \|G_{\alpha,r}\|_\infty \leq K_0$,
      so the map
      $
        \alpha \mapsto P_\alpha h = \sum_r p_{\alpha,r}
      $
      is continuous from $[\alpha_-, \alpha_+]$ to $C^2$, and
      $\|P_\alpha h\|_{C^2} \leq K_0(4 K_0+1) \|h\|_{C^2}$.
    \item
      Let $q_{\alpha,r} = \partial_\alpha [(h \circ F_{\alpha,r}^{-1}) G_{\alpha,r}]$.
      We use the fact that
      $F_{\alpha,r}^{-1}$, as a function of $\alpha$ and $\xi$,
      has continuous partial derivatives up to second order, to compute
      \[
        q_{\alpha,r}' = \partial_\alpha \bigl[((h \circ F_{\alpha,r}^{-1})G_{\alpha,r})'\bigr]
        = \partial_\alpha \left[ \pm (h'\circ F_{\alpha,r}^{-1}) G_{\alpha,r}^2 
                              + (h\circ F_{\alpha,r}^{-1}) \, G'_{\alpha,r} \right].
      \]
      We use assumptions \ref{A1}, \ref{A2}, \ref{A4}, \ref{A5} and \ref{A6}
      to estimate
      \[
        \|q_{\alpha,r}\|_{C^{1}} \leq \|G_{\alpha,r}\|_\infty (4+K_{0}) \gamma_r \|h\|_{C^2}.
      \]
      Since $q_{\alpha,r}$ and $q_{\alpha,r}'$
      are jointly continuous in $\alpha$ and $\xi$,
      we obtain that the map $\alpha \mapsto q_{\alpha,r}$ from $[\alpha_-, \alpha_+]$
      to $C^{1}(I)$ is continuous.
      
      By assumption \ref{A7}, $\sum_r \gamma_r \|G_{\alpha,r}\|_\infty \leq K_0$,
      so the map
      $
        \alpha \mapsto Q_\alpha h = \sum_r q_{\alpha,r}
      $
      is continuous from $[\alpha_-, \alpha_+]$ to $C^{1}$, and
      $\|Q_\alpha h\|_{C^{1}} \leq K_0(4+K_0) \|h\|_{C^2}$.
    \item
      Note that $(Q_\alpha h)(\xi)$ and
      $(Q_\alpha h)'(\xi)$ are jointly continuous in 
      $\alpha$ and $\xi$. By definition of $Q_\alpha$, 
      for every $\xi$ and $j=0,1$ we can write
      \begin{align*}
        (P_\beta h)^{(j)}(\xi) - (P_\alpha h)^{(j)}(\xi)
        &= \int_\alpha^\beta (Q_t h)^{(j)}(\xi) \, dt \\
        &= (\beta - \alpha) (Q_\alpha h)^{(j)} (\xi) 
          + \int_\alpha^\beta \bigl[(Q_t h)^{(j)}(\xi) - (Q_\alpha h)^{(j)} (\xi) \bigr] \, dt.
      \end{align*}
      Fix $\alpha$. Since $\lim_{t \to \alpha} \|Q_t h - Q_\alpha h\|_{C^{1}} = 0$, 
      the integral on the right is
      $o(\beta - \alpha)$ uniformly in $\xi$. Therefore,
      \[
        \| P_\beta h - P_\alpha h - (\beta - \alpha) Q_\alpha h \|_{C^{1}} 
        = o(\beta - \alpha),
      \]
      thus $\cQ_h \colon [\alpha_-, \alpha_+] \to C^{1}(I)$
      is the derivative of $\cP_h [\alpha_-, \alpha_+] \to C^{1}(I)$.
      
    \item
      To prove that $\int Q_\alpha h \, dm = 0$, we differentiate the identity
      $\int P_\alpha h \, dm = \int h \, dm$ by $\alpha$:
      \[
        \int Q_\alpha h \, dm
        = \frac{d}{d \alpha} \int P_\alpha h \, dm
        = 0,
      \]
      the order of differentiation and integration can be changed because both
      $P_\alpha h$ and $ \partial_\alpha (P_\alpha h) = Q_\alpha h$ are jointly 
      continuous in $\alpha$ and $\xi$.
  \end{enumerate}
\end{proof}

\subsection{Distortion bounds and coupling}
\label{subsection:distortionsandcoupling}
If $h \in C^1$ and $h$ is positive, denote $\| h \|_{L} = \| h'/h \|_\infty$.
If also $h \in C^2$, denote $\|h\|_P= \|h''/h\|_\infty$.

\begin{lemma}[Distortion bounds]
  \label{lemma:lyineq}
  If $h \in C^1$ and $h>0$, then
  \begin{equation}
    \label{eq:distortions1}
  \|P_\alpha h\|_L \leq  \sigma \|h\|_L + K_0.
  \end{equation}
  If also $h \in C^2$, then
  \begin{equation}
    \label{eq:distortions2}
    \|P_\alpha h\|_P \leq \, \sigma^2 \|h\|_P  + 3 \sigma K_0 \|h\|_L 
    + K_0.
  \end{equation}
\end{lemma}
\begin{proof}
Recall that 
$P_\alpha h = \sum_r  (h \circ F_{\alpha,r}^{-1}) G_{\alpha,r}$
and $(F_{\alpha,r}^{-1})' = \pm G_{\alpha,r}$, where the sign depends only on $r$.
Inequality (\ref{eq:distortions1}) follows from the following computation:
\begin{equation*}
  \begin{aligned}
  \left|\frac{(P_\alpha h)'}{P_\alpha h}\right| 
  &= \left| \frac{\sum_r \pm (h'\circ F_{\alpha,r}^{-1}) G_{\alpha,r}^2 
                          + (h\circ F_{\alpha,r}^{-1}) \, G'_{\alpha,r}}
                {\sum_r (h\circ F_{\alpha,r}^{-1}) G_{\alpha,r}}
    \right| \\
  &\leq \max_r \left| \frac{\pm (h'\circ F_{\alpha,r}^{-1}) G_{\alpha,r}^2 
                          + (h\circ F_{\alpha,r}^{-1}) \, G'_{\alpha,r}}
                          {(h\circ F_{\alpha,r}^{-1}) G_{\alpha,r}}
    \right| 
  \leq \max_r\left(
  \frac{|h'\circ F_{\alpha,r}^{-1}|}{h\circ F_{\alpha,r}^{-1}} G_{\alpha,r}
  + \frac{|G_{\alpha,r}'|}{G_{\alpha,r}}
  \right) \\
  &\leq \max_r \left( \|h\|_L \|G_{\alpha,r}\|_\infty + \|G_{\alpha,r}\|_L \right)
  \end{aligned}
\end{equation*}
and assumptions \ref{A1} and \ref{A2}.

Next,
\begin{equation*}
  \label{eq:predistortions2}
  (P_\alpha h)'' = \sum_r \left[
    (h''\circ F_{\alpha,r}^{-1}) G_{\alpha,r}^3
    \pm 3 (h'\circ F_{\alpha,r}^{-1}) G_{\alpha,r} G_{\alpha,r}'
    + (h\circ F_{\alpha,r}^{-1}) G_{\alpha,r}''
    \right].
\end{equation*}
Thus
\begin{equation*}
  \begin{aligned}
  \left|\frac{(P_\alpha h)''}{P_\alpha h} \right| 
  &\leq \max_r
    \frac{ \left| (h''\circ F_{\alpha,r}^{-1}) G_{\alpha,r}^3
      \pm 3 (h'\circ F_{\alpha,r}^{-1}) G_{\alpha,r} G_{\alpha,r}'
      + (h\circ F_{\alpha,r}^{-1}) G_{\alpha,r}'' \right|
         }{(h\circ F_{\alpha,r}^{-1}) G_{\alpha,r}} \\
  &\leq \max_r \left(
    \frac{|h''\circ F_{\alpha,r}^{-1}|}{h\circ F_{\alpha,r}^{-1}}
    G_{\alpha,r}^2
    + 3 \frac{|h'\circ F_{\alpha,r}^{-1}|}{h\circ F_{\alpha,r}^{-1}}
    \frac{|G_{\alpha,r}'|}{G_{\alpha,r}} G_{\alpha,r}
    +  \frac{|G_{\alpha,r}''|}{G_{\alpha,r}}
  \right) \\
  &\leq \max_r \left(
    \|h\|_P \, \|G_{\alpha,r}\|_\infty^2 + 3 \|h\|_L\, \|G_{\alpha,r}\|_L\, \|G_{\alpha,r}\|_\infty + 
    \|G_{\alpha,r}\|_P
  \right).
  \end{aligned}
\end{equation*}
The inequality (\ref{eq:distortions2}) follows from the above and assumptions
\ref{A1}, \ref{A2} and \ref{A3}.
\end{proof}

Let $K_L>0$, $K_P>0$ and $\theta\in (0,1)$ be constants satisfying
\begin{equation}
  \label{eq:KLKP}
  \begin{aligned}
  K_L(1-\theta \, e^{|I| K_L}) & > \sigma \, K_L + K_0 , \quad \text{and} \\
  K_P(1-\theta \, e^{|I| K_L}) & > \sigma^2 \, K_P  + 3 \sigma K_0 \, K_L 
  + K_0,
  \end{aligned}
\end{equation}
where $|I|$ means the length of the interval $I$.
It is clear that such constants can be chosen, because
$\sigma < 1$.

\begin{definition}
  We say that a function $h$ is \emph{regular}, if it is positive,
  belongs in $C^1(I)$, and
  $\|h\|_L \leq K_L$. If in addition $h\in C^2(I)$ and
  $\|h\|_P \leq K_P$, we say that $h$ is \emph{superregular}.
\end{definition}

\begin{remark}
It readily follows from Lemma \ref{lemma:lyineq}
that if $h$ is a regular function, then so is $P_\alpha h$. If $h$ is
superregular, then so is $P_\alpha h$. 
\end{remark}

\begin{remark}
We observe that regular functions are explicitly bounded from above and from below.
If $h$ is regular, so $\|h'/h\|_\infty < K_L$, then
$h(x_1)/h(x_2)\le e^{|I| K_L}$ for all 
$x_1,x_2$. Also, there is $\hat{x} \in I$ such that
$h(\hat{x}) = \int h \, dm$, hence
\begin{equation}
  \label{eq:regbounds}
   \SMALL e^{-|I| K_L} \int h \, dm \leq h(x) \leq e^{|I| K_L} \int h \, dm 
   \qquad \text {for all } x \in I.
\end{equation}
\end{remark}

\begin{lemma}
  Assume that h is regular. Let 
  $g = P_\alpha h - \theta \, \int h \, dm$. Then
  $g$ is regular. If $h$ is superregular, then so is $g$.
\end{lemma}
\begin{proof}
  Since $P_\alpha h \geq e^{-|I| K_L} \int P_\alpha h \, dm = e^{-|I| K_L} \int h \, dm$
  by equation (\ref{eq:regbounds}),
  \[
    g = P_\alpha h\left(1-\frac{\theta \, \int h \, dm}{P_\alpha h}\right)
    \geq P_\alpha h \left( 1 - \theta e^{|I| K_L} \right).
  \]
  Thus $g > 0$. By Lemma \ref{lemma:lyineq} and equation (\ref{eq:KLKP})
  \[
    \|g\|_L 
    = \left\|\frac{g'}{g}\right\|_\infty 
    = \left\|\frac{(P_\alpha h)'}{g}\right\|_\infty 
    \leq 
    \left\|\frac{(P_\alpha h)'}{P_\alpha h}\right\|_\infty \frac{1}{1-\theta e^{|I| K_L}}
    = \|P_\alpha h\|_L \frac{1}{1-\theta e^{|I| K_L}}
    \leq K_L.
  \]
  Hence $g$ is regular.
  An analogous proof works for $\|g\|_P$.
\end{proof}

\begin{lemma}[Coupling Lemma]
  \label{lemma:coupling}
  Let $f$ and $g$ be two regular functions with $\int f \, dm = \int g \, dm = M$.
  Let $f_0 = f$ and $g_0 = g$, and define
  \[ \SMALL
    f_{n+1}=P_\alpha f_n - \theta \int f_n \, dm, \qquad 
    g_{n+1}=P_\alpha g_n - \theta \int g_n \, dm.
  \]
  Then for all $n$
  \[
    P_\alpha^n(f-g) = f_n - g_n,
  \]
  where $f_n$ and $g_n$ are regular, and 
  $\int f_n \, dm = \int g_n \, dm = (1-\theta)^n M$.
  
  In particular, $\|f_n\|_\infty, \|g_n\|_\infty \leq
  (1-\theta)^n e^{|I| K_L} M$, and
  \[
    \|f_n' \|_\infty,\|g_n' \|_\infty \leq K_L (1-\theta)^n e^{|I| K_L} M.
  \]
  If in addition $f$ and $g$ are superregular, then
  \[
    \|f_n''\|_\infty,\|g_n''\|_\infty \leq K_P (1-\theta)^n e^{|I| K_L} M.
  \]
\end{lemma}
\begin{proof}
  The proof of $\int f_n \, dm = \int g_n \, dm = (1-\theta)^n M$ is by induction.
  
  By equation (\ref{eq:regbounds}), $\|f\|_\infty$ and $\|g\|_\infty$ are bounded by
  $(1-\theta)^n e^{|I| K_L} M$. Note that if $h$ is a regular function, then
  $\|h' \|_\infty \leq K_L \|h\|_\infty$, and if it is superregular, then also
  $\|h''\|_\infty \leq K_P \|h\|_\infty$.
  The bounds on 
  $\|f'\|_\infty$, $\|g'\|_\infty$, $\|f''\|_\infty$, $\|g''\|_\infty$ follow.
\end{proof}

\begin{corollary}
  \label{cor:couplc}
  There is a constant $K_{5}$ such that if $h \in C^i (I)$ 
  for $i=1 \text{ or } 2$, and $h$ has mean zero, then 
  \[
    \|P_\alpha^n h\|_{C^i} \leq K_{5} (1-\theta)^n \, \|h\|_{C^i}.
  \]
\end{corollary}
\begin{proof}
  We can represent $h = (h+c) - c$, where 
  $c= \|h\|_{C^i} (1 + \max(K_L^{-1}, K_P^{-1}))$. 
  Then 
  \[
    \left\|\frac{h'}{h+c}\right\|_\infty \leq \frac{\|h\|_{C^i}}{-\|h\|_{C^i} + c}
    = \frac{1}{\max (K_L^{-1}, K_P^{-1})}
    = \min(K_L, K_P),
  \]
  and so $h+c$ is regular. If $i=2$, then the same identity with $h''$ in 
  place of $h'$ also holds true, so also $h+c$ is superregular.
  
  By Lemma \ref{lemma:coupling} applied to $f= h+c$ and $g=c$,
  \[
    P_\alpha^n h = f_n - g_n,
  \]
  where 
  \begin{align*} 
    \|f_n\|_{C^i}, \|g_n\|_{C^i} 
    &\leq \max(1, K_L, K_P) (1-\theta)^n e^{|I| K_L} c \\
    &= (1-\theta)^n \left[ \max(1, K_L, K_P)  e^{|I| K_L} (1 + \max(K_L^{-1}, K_P^{-1})) \right] \|h\|_{C^i} \\
    &= (1-\theta)^n \frac{K_{5}}{2} \|h\|_{C^i}.
  \end{align*}
  Thus $\|P_\alpha^n h\|_{C^i} \leq (1-\theta)^n K_{5} \|h\|_{C^i}$, where
  \[
    K_{5} = 2 \max(1, K_L, K_P) (1 + \max(K_L^{-1}, K_P^{-1})) \, e^{|I| K_L}.
  \]
\end{proof}

\begin{corollary}
  \label{cor:coupla} 
  For any $h \in C^2 (I)$ with $\int h \, dm = 1$
  \[
    h_\alpha = \lim_{n \to \infty} P_\alpha^n h
    = h + \sum_{n=0}^\infty P_\alpha^n(P_\alpha h-h).
  \]
  The series converges exponentially fast in $C^2$.
  The $C^2$ norm of $h_\alpha$ is bounded by 
  $K_{1} = 1 + 2 \theta^{-1} e^{|I| K_L} \max(1, K_L, K_P)$.
\end{corollary}
\begin{proof}
  Let
  \[
    f = 1 + \sum_{n=0}^\infty P_\alpha^n(P_\alpha 1-1).
  \]
  Since $1$ is a superregular function, so is $P_\alpha 1$, and
  by Lemma \ref{lemma:coupling} applied to $f=P_\alpha 1$ and $g=1$,
  we have that 
  $\|P_\alpha^n(P_\alpha1-1)\|_{C^2} \le 2 (1-\theta)^n \, e^{|I| K_L} \max(1, K_L, K_P)$. 
  Thus the series above converges exponentially fast in $C^2 (I)$
  and $\|h\|_{C^2} \leq K_{1}$.
  
  Since $1 + \sum_{n=0}^N P_\alpha^n(P_\alpha 1-1) = P_\alpha^{N+1} 1$, 
  we have $f = \lim_{n \to \infty} P_\alpha^n 1$. Thus
  $f$ is invariant under $P_\alpha$. It is clear that $\int f \, dm = 1$.
  Thus $h_\alpha = f$.
  
  By Corollary \ref{cor:couplc}, the $C^2$ norm of $P_\alpha^n (h - h_\alpha) = (P_\alpha^n h) - h_\alpha$
  decreases exponentially with $n$, thus $h_\alpha = \lim_{n\to\infty} P_\alpha^n h
  = h + \sum_{n=0}^\infty P_\alpha^n(P_\alpha h-h)$.
\end{proof}

\begin{corollary}
  \label{cor:halphacunt}
  The map $\alpha \mapsto h_\alpha$ from 
  $[\alpha_-, \alpha_+]$ to $C^2(I)$ is continuous.
  
\end{corollary}
\begin{proof}
  Using Corollary \ref{cor:coupla}, write for $N \in \bN$:
  \[
    h_\alpha = 1
    + \sum_{n=0}^{N-1} P_\alpha^n(P_\alpha 1-1)
    + \sum_{n=N}^{\infty} P_\alpha^n(P_\alpha 1-1).
  \]
  The $C^2$ norm of the second sum is exponentially small in $N$, 
  uniformly in $\alpha$.
  By Lemma~\ref{lemma:diffalpha1}, a map $\alpha \mapsto P_\alpha^n(P_\alpha h-h)$
  from $[\alpha_-, \alpha_+]$ to $C^2(I)$ is continuous for every $n$.
  Thus the first sum depends on $\alpha$ continuously. 
  Since the choice of $N$ is arbitrary, the result follows.
\end{proof}

\section{Proofs of Theorems \ref{th:lsv} and \ref{th:whatamIdoing}}
  \label{section:LSV}

In this section we prove technical statements about the family of maps 
$T_\alpha$, defined by equation (\ref{eq:LSVT}). 
We use notations introduced in Section \ref{section:LSVintro}. 

In Subsection \ref{section:technicallemmas} 
we introduce necessary notations and prove a number of technical lemmas, 
in Subsections \ref{subsection:proofofth:lsv} and 
\ref{subsection:proofofth:whatamIdoing} we use the accumulated knowledge
to prove Theorems~\ref{th:lsv} and \ref{th:whatamIdoing}.

\subsection{Technical Lemmas}
\label{section:technicallemmas}

We use notation $\const$ for various nonnegative constants, 
which only depend on $\alpha_-$ and $\alpha_+$,
and may change from line to line, and within one expression if used twice.
Recall the definition of \(y_r\) from the beginning of 
Section~\ref{section:LSVintro}.

It is clear that $T_\alpha(x)$, as a function of $\alpha$
and $x$, has continuous partial derivatives of all orders in
$\alpha,x \in [\alpha_-,\alpha_+] \times (0, 1/2]$, and so do
$F_{\alpha,r}(x)$ and $F_{\alpha,r}^{-1}(x)$ on
$[\alpha_-,\alpha_+] \times [y_{r+1}, y_r]$  and $[\alpha_-,\alpha_+] \times [1/2, 1]$ 
respectively.

Let $E_{\alpha}\colon [0, 1/2] \to [0,1]$, $E_\alpha x = T_\alpha x$ be the left branch
of the map $T_\alpha$. Note that $E_\alpha$ is invertible.
Let $z \in [0,1]$ and write,
for notational convenience,
$z_r = E_{\alpha}^{-r}(z)$. Then
$F_{\alpha,r} (z) = E_\alpha^r (T_\alpha (z)) = E_\alpha^r(2 z -1)$ 
for $z \in [y_{r+1},y_r]$, and
for $z \in [1/2,1]$
\begin{equation}
  \label{eq:l93n}
  T_\alpha(F_{\alpha,r}^{-1}(z)) =  2 F_{\alpha,r}^{-1}(z) - 1 = z_r.
\end{equation}
By $(\cdot)'$ we denote the derivative with respect to $z$.
Let $G_{\alpha,r}$ be defined as in Theorem~\ref{th:dhdalpha}. Then 
for $z \in [1/2,1]$
\begin{equation}
  \label{eq:l93m}
  G_{\alpha,r}(z) = (F_{\alpha,r}^{-1})'(z) = z_r'/2.
\end{equation}
We do all the analysis in terms of $z_r$, and the relation to $G_{\alpha,r}$
and $F_{\alpha,r}^{-1}$ is given by equations
(\ref{eq:l93n}) and (\ref{eq:l93m}).

\begin{remark}
  \label{remark:zzzzzz}
  By construction, $z_0 = z$, $z_0'=1$ and $z_0''=0$. Note that 
  $z_r\leq 1/2$ for $r\geq 1$. Also
  \begin{align}
    z_r &= z_{r+1}(1+2^\alpha z_{r+1}^\alpha), 
    \label{eq:hazdd} \\
    z_r' &= [1 + (\alpha+1) 2^\alpha z_{r+1}^\alpha] z_{r+1}',
    \label{eq:haztt} \\
    z_r' &= \prod_{j=1}^r \left[1+(\alpha+1) 2^\alpha z_j^\alpha\right]^{-1}.
    \label{eq:hazzz}
  \end{align}
\end{remark}

Our analysis is built around the following estimate:
\begin{lemma}
  \label{lemma:Iansestimate}
  For $r \geq 1$
  \[
    \frac{1}{z_0^{-\alpha} + r\, \alpha 2^\alpha}
    \leq z_r^{\alpha}
    \leq \frac{1}{z_0^{-\alpha} + r \,  \alpha (1-\alpha) 2^{\alpha-1}}.
  \]
  In particular,
  \[
    \frac{\const \, z_0^\alpha}{r}\leq z_r^\alpha \leq \frac{\const}{r}
    \qquad \text{and} \qquad
    -\log z_r \leq \const \, \left[ \logg r - \log z_0 \right].
  \]
\end{lemma}
\begin{proof}
  Transform equation (\ref{eq:hazdd}) into
  \[
    z_{r+1}^{-\alpha} = z_r^{-\alpha} + \frac{1-(1+2^\alpha z_{r+1}^\alpha)^{-\alpha}}{z_{r+1}^\alpha}.
  \]
  Then
  \begin{equation}
    \label{eq:wg93}
    z_r^{-\alpha} = z_0^{-\alpha} 
    + \sum_{j=1}^r \frac{1-(1+2^\alpha z_{j}^\alpha)^{-\alpha}}{z_{j}^\alpha}.
  \end{equation}
  
  For all $t \in (0,1)$ and all $\alpha \in (0,1)$
  \[
    1-\alpha t \leq (1+t)^{-\alpha} \leq 1 - \alpha t + \frac{\alpha (\alpha+1)}{2} t^2.
  \]
  Since $z_j \in (0,1/2]$ for $j \geq 1$, using the above inequality with
  $t=2^\alpha z_j^\alpha$, we obtain
  \[
    \alpha (1-\alpha) 2^{\alpha-1} 
    \leq \frac{1-(1+2^\alpha z_{j}^\alpha)^{-\alpha}}{z_{j}^\alpha} 
    \leq \alpha 2^\alpha.
  \]
  By equation (\ref{eq:wg93}),
  \[
    r \, \alpha (1-\alpha)  2^{\alpha-1} 
    \leq z_r^{-\alpha} - z_0^{-\alpha}
    \leq r \, \alpha 2^\alpha.
  \]
  for $r \geq 1$. Write
  \[
    \frac{z_0^\alpha}{r} \frac{1}{1 + \alpha 2^\alpha}
    \leq \frac{z_0^\alpha}{r} \frac{1}{r^{-1} + z_0^{\alpha} \, \alpha 2^\alpha}
    = \frac{1}{z_0^{-\alpha} + r\, \alpha 2^\alpha}
    \leq z_r^{\alpha}
    \leq \frac{1}{z_0^{-\alpha} + r \,  \alpha (1-\alpha) 2^{\alpha-1}}.
  \]
  The result follows.
\end{proof}

\begin{lemma}
  \label{lemma:Izk}
  $z_0'=1$ and
  \[
    0 \leq z_r' \leq \const \, \left(1+ r z_0^{\alpha} \alpha 2^\alpha \right)^{-(\alpha+1)/\alpha}
    \leq \const \, r^{- (\alpha+1)/\alpha} \, z_0^{-(\alpha+1)}
  \]
  for $r\geq 1$.
\end{lemma}
\begin{proof}
  By Remark \ref{remark:zzzzzz}, $z_0'=1$. Let $r \geq 1$.
  Using the inequality
  \[
    \frac{1}{1+t} \leq \exp(-t+t^2) \quad \text{for} \quad t \geq 0
  \]
  on equation (\ref{eq:hazzz}) we obtain
  \begin{equation}
    \label{eq:tasc3}
    0 \leq z_r' = \prod_{j=1}^r \frac{1}{1+(\alpha+1)2^\alpha z_j^\alpha} 
    \leq \exp \left( -\sum_{j=1}^{r} (\alpha+1) 2^\alpha z_j^\alpha 
    +\sum_{j=1}^{r} \left((\alpha+1) 2^{\alpha} z_j^{\alpha}\right)^2
    \right).
  \end{equation}
  By Lemma \ref{lemma:Iansestimate}, 
  $(z_j^\alpha)^2 \leq \const/j^2$, thus
  the second sum under the exponent is bounded by $\const$.
  Also by Lemma \ref{lemma:Iansestimate},
  \begin{align*}
    \sum_{j=1}^{r} z_j^\alpha
    & \geq \sum_{j=1}^r \frac{1}{z_0^{-\alpha} + j \alpha 2^\alpha}
      \geq \int_1^r \frac{z_0^\alpha}{1 + t z_0^{\alpha} \alpha 2^\alpha} \, dt - \const \\
    & = \frac{1}{\alpha 2^\alpha} 
      \log \left( 1 + t z_0^{\alpha} \alpha 2^\alpha\right) \big|_{t=1}^{t=r}
      - \const \\
    & \geq \frac{\log \left( 1+ r z_0^{\alpha} \alpha 2^\alpha \right)}{\alpha 2^\alpha} 
      - \const
  \end{align*}
  Thus
  \[
    - (\alpha+1) 2^\alpha \sum_{j=1}^{r} z_j^\alpha 
    \leq - \frac{\alpha+1}{\alpha} \log \left(1+ r z_0^{\alpha} \alpha 2^\alpha \right) 
      + \const,
  \]
  and by equation (\ref{eq:tasc3}),
  \[
    z_r' 
    \leq \const \, \left(1+ r z_0^{\alpha} \alpha 2^\alpha \right)^{-(\alpha+1)/\alpha}
    \leq \const \, \left(r z_0^{\alpha} \alpha 2^\alpha \right)^{-(\alpha+1)/\alpha}
    \leq \const \, r^{-(\alpha+1)/\alpha} z_0^{-(\alpha+1)}.
  \]
  \end{proof}

\begin{lemma}
  \label{lemma:Ipzk}
  $ 0 \leq -z_r''/ z_r' \leq \const \, z_0^{-2} / \max(r,1)$.
\end{lemma}
\begin{proof}
  Differentiating both sides of the equation
  (\ref{eq:haztt}), we obtain
  \begin{equation}
    \label{eq:dz2}
    z_r'' = \alpha(\alpha+1) 2^\alpha z_{r+1}^{\alpha-1} (z_{r+1}')^2 
    + (1+(\alpha+1)2^\alpha z_{r+1}^\alpha) z_{r+1}''.
  \end{equation}
  Dividing the above by $z_r' = [1 + (\alpha+1) 2^\alpha z_{r+1}^\alpha] z_{r+1}'$
  we get
  \[
    \frac{z_r''}{z_r'} = \frac{\alpha(\alpha+1) 2^\alpha z_{r+1}^{\alpha-1} z_{r+1}'}{
      1 + (\alpha+1) 2^\alpha z_{r+1}^\alpha}
    + \frac{z_{r+1}''}{z_{r+1}'}.
  \]
  Recall that $z_0''/z_0' = 0$ and $z_r' \geq 0$, thus $z_r'' \leq 0$ for all $r$.
  By Lemmas \ref{lemma:Iansestimate} and \ref{lemma:Izk} we have
  \[
    0 \leq z_r^{\alpha-1} z_r' 
    \leq \const \, \left(\frac{z_0^{\alpha}}{r}\right)^{(\alpha-1)/\alpha}
      r^{- \frac{\alpha+1}{\alpha}} \, z_0^{-(\alpha+1)}
    \leq \const \, r^{-2} \, z_0^{-2}
  \]
  for $r \geq 1$. Thus
  \[
    0 \leq \frac{z_r''}{z_r'} - \frac{z_{r+1}''}{z_{r+1}'}
    \leq \const \, (r+1)^{-2} \, z_0^{-2}.
  \]
  The result follows.
\end{proof}

\begin{lemma}
  \label{lemma:Ippzk}
  $|z_r'''/ z_r' | \leq \const \, z_0^{-\alpha - 4} / \max(r^2,1)$.
\end{lemma}
\begin{proof}
  Differentiate the equation (\ref{eq:dz2}).
  This results in
  \begin{align*}
    z_r'''  = (\alpha-1) \alpha (\alpha+1) 2^\alpha z_{r+1}^{\alpha-2} (z_{r+1}')^3
     + 3 \alpha (\alpha+1) 2^\alpha z_{r+1}^{\alpha-1} z_{r+1}'z_{r+1}'' \\
     + (1+(\alpha+1) 2^\alpha z_{r+1}^\alpha) z_{r+1}'''.
  \end{align*}
  Dividing the above by $z_r' = [1 + (\alpha+1) 2^\alpha z_{r+1}^\alpha] z_{r+1}'$
  we get
  \begin{equation*}
    \begin{aligned}
    \frac{z_r'''}{z_r'}
    = \frac{(\alpha-1) \alpha (\alpha+1) 2^\alpha z_{r+1}^{\alpha-2} (z_{r+1}')^2}
           {1 + (\alpha+1) 2^\alpha z_{r+1}^\alpha}
     +\frac{3 \alpha (\alpha+1) 2^\alpha z_{r+1}^{\alpha-1} z_{r+1}'}
           {1 + (\alpha+1) 2^\alpha z_{r+1}^\alpha} \frac{z_{r+1}''}{z_{r+1}'}
     +\frac{z_{r+1}'''}{z_{r+1}'}.
    \end{aligned}
  \end{equation*}
  Using Lemmas~\ref{lemma:Iansestimate}, \ref{lemma:Izk} and \ref{lemma:Ipzk} we bound the
  first two terms in the right hand side above by $\const \, (r+1)^{-3} \, z_0^{-\alpha - 4}$ 
  and $\const \, (r+1)^{-3} \, z_0^{-4}$ respectively.
  Thus 
  \[
    \left| \frac{z_r'''}{z_r'} - \frac{z_{r+1}'''}{z_{r+1}'} \right| 
    \leq \const \, (r+1)^{-3} \, z_0^{-\alpha - 4}.
  \]
  Since $z_0'''/z_0' = 0$, the result follows.
\end{proof}

\begin{lemma}
  \label{lemma:dzalpha}
  $\partial_\alpha z_0 = 0$ and for $r \geq 1$
  \begin{align*}
    0 \leq \frac{\partial_\alpha z_r}{z_r} 
    &\leq \const \, \logg (r z_0^\alpha) \, \left[ \logg r -\log z_0 \right]
    \qquad \text{and} \\
    0 \leq \partial_\alpha z_r 
    & \leq \const \, \frac{\logg (r z_0^\alpha)}{r^{1/\alpha}}
      \left[ \logg r -\log z_0 \right].
  \end{align*}
\end{lemma}

\begin{proof}
Since $z_0 = z$ does not depend on $\alpha$, $\partial_\alpha z_0 = 0$.

Differentiating the identity $z_{r+1}(1+2^\alpha z_{r+1}^\alpha)=z_r$ by $\alpha$ we
obtain a recursive relationz
\[
  \partial_\alpha z_{r+1} = \frac{ \partial_\alpha z_{r} + 2^\alpha z_{r+1}^{\alpha+1} (-\log(2z_{r+1}))}
  {1+(\alpha+1) 2^\alpha z_{r+1}^\alpha}.
\]
Since $z_{r+1} \leq 1/2$ for all $r$, it follows that $\partial_\alpha z_r \geq 0$ for all $r$.
It is convenient to rewrite the above, dividing by $z_{r+1}$ and using
$z_{r+1}(1+2^\alpha z_{r+1}^\alpha)=z_r$:
\[
  \frac{\partial_\alpha z_{r+1}}{z_{r+1}} = 
  \frac{(1+2^\alpha z_{r+1}^\alpha)
  \frac{\partial_\alpha z_{r}}{z_r} + 2^\alpha z_{r+1}^\alpha (-\log(2z_{r+1}))}{
    1+(\alpha+1) 2^\alpha z_{r+1}^\alpha},
\]
which implies
\[
  \frac{\partial_\alpha z_{r+1}}{z_{r+1}} \leq \frac{\partial_\alpha z_{r}}{ z_{r}}
  + 2^\alpha z_{r+1}^\alpha (-\log (2 z_{r+1})).
\]
By Lemma \ref{lemma:Iansestimate},
\[
  2^\alpha z_r^\alpha (-\log (2 z_r)) 
  \leq \const \, \frac{\logg r - \log z_0}{z_0^{-\alpha} + r \alpha(1-\alpha) 2^{\alpha-1}}.
\]
Hence
\begin{equation}
  \label{eq:11076}
  \begin{aligned}
    \frac{\partial_\alpha z_{r}}{z_{r}} 
    & \leq \sum_{j=1}^{r} 2^\alpha z_j^\alpha (-\log (2 z_j))
      \leq \const \, \int_1^r \frac{\logg t - \log z_0}{
      z_0^{-\alpha} + t \alpha(1-\alpha)2^{\alpha-1}} \, dt \\
    & \leq \const \, \logg (r z_0^\alpha) \left[ \logg (r z_0^\alpha) - \log z_0 \right].
  \end{aligned}
\end{equation}
The first part of the lemma follows. To prove the second part,
observe that by Lemma~\ref{lemma:Iansestimate}, 
$z_r \leq \const \, r^{-1/\alpha}$.
\end{proof}

\begin{lemma}
  \label{lemma:dIalpha}
  $\left|(\partial_\alpha z_r')/z_r'\right| 
  \leq \const \, \left[\logg (r z_0^\alpha) \right]^2 \, \left[ \logg r - \log z_0 \right]$.
\end{lemma}
\begin{proof}
  Note that $\partial_\alpha z_0' = 0$, because $z_0=z$ does not depend on $\alpha$.
  
  Differentiate equation (\ref{eq:haztt}) by $\alpha$. This results in
  \begin{align*}
    \partial_\alpha z_r'
    =
    \left(
      2^\alpha z_{r+1}^\alpha + (\alpha+1) 2^\alpha z_{r+1}^\alpha \log(2z_{r+1})
      + \alpha (\alpha+1) 2^\alpha z_{r+1}^{\alpha-1} \partial_\alpha z_{r+1}
    \right) z_{r+1}' \\
    + (1+(\alpha+1) 2^\alpha z_{r+1}^\alpha) \partial_\alpha z_{r+1}'
  \end{align*}
  Dividing the above by $z_r' = [1 + (\alpha+1) 2^\alpha z_{r+1}^\alpha] z_{r+1}'$
  we get
  \begin{equation*}
    \label{eq:lka67}
    \frac{\partial_\alpha z_r'}{z_r'}
    =
    \frac{
      2^\alpha z_{r+1}^\alpha + (\alpha+1) 2^\alpha z_{r+1}^\alpha \log(2z_{r+1})
      + \alpha (\alpha+1) 2^\alpha z_{r+1}^{\alpha-1} \partial_\alpha z_{r+1}
         }{1 + (\alpha+1) 2^\alpha z_{r+1}^\alpha}
    + \frac{\partial_\alpha z_{r+1}'}{z_{r+1}'}.
  \end{equation*}
  For $r\geq 1$ Lemmas \ref{lemma:Iansestimate} and \ref{lemma:dzalpha} give
  $|z_r^\alpha | \leq \const / (z_0^{-\alpha} + r \alpha(1-\alpha) 2^{\alpha-1})$, 
  \begin{align*}
    | z_r^\alpha \log z_r |
    &\leq \const \, \frac{\logg r - \log z_0}{z_0^{-\alpha} + r \alpha(1-\alpha) 2^{\alpha-1}}
    \qquad \text{and} \\
    |z_r^{\alpha-1} \partial_\alpha z_r| 
    = \left|z_r^{\alpha} \frac{\partial_\alpha z_r}{z_r}\right|
    & \leq \const \, \frac{\logg (r z_0^\alpha) \, ( \logg r  -\log z_0 )}{
      z_0^{-\alpha} + r \alpha(1-\alpha) 2^{\alpha-1}}.
  \end{align*}
  Therefore 
  \[
    \left| \frac{\partial_\alpha z_r'}{z_r'} - \frac{\partial_\alpha z_{r+1}'}{z_{r+1}'} \right|
    \leq \const \, \frac{\logg ((r+1) z_0^\alpha) \, ( \logg (r+1)  -\log z_0 )}{
      z_0^{-\alpha} + (r+1) \alpha(1-\alpha) 2^{\alpha-1}}.
  \]
  Thus
  \[
    \left|\frac{\partial_\alpha z_r'}{z_r'}\right| 
    \leq \const \, \int_1^r 
      \frac{\logg (t z_0^\alpha) \, ( \logg t  -\log z_0 )}{
        z_0^{-\alpha} + t \alpha(1-\alpha) 2^{\alpha-1}} \, dt
    \leq \const \, \left[ \logg (r z_0^\alpha) \right]^2 \left( \logg r - \log z_0 \right).
  \]
\end{proof}

\begin{lemma}
  \label{lemma:dIpalpha}
  $\left|(\partial_\alpha z_r'')/z_r' \right|
  \leq \const \, z_0^{-2} \, (1 - \log z_0 )$.
\end{lemma}
\begin{proof}
  Differentiate both sides of the equation (\ref{eq:dz2}) by $\alpha$. This gives
  \begin{equation*}
    \begin{aligned}
    \partial_\alpha z_r'' = &[2 \alpha + 1 + \alpha (\alpha + 1) \log(2 z_{r+1})]
       2^\alpha z_{r+1}^{\alpha-1} (z_{r+1}')^2 \\
     & + (\alpha-1) \alpha (\alpha + 1) 2^\alpha z_{r+1}^{\alpha-2} (z_{r+1}')^2 \partial_\alpha z_{r+1}
       + 2 \alpha (\alpha + 1) 2^\alpha z_{r+1}^{\alpha-1} z_{r+1}' \partial_\alpha z_{r+1}' \\
     & + (1 + (\alpha+1) \log(2 z_{r+1}))2^\alpha z_{r+1}^\alpha z_{r+1}''
       + \alpha (\alpha+1) 2^\alpha z_{r+1}^{\alpha-1} z_{r+1}'' \partial_\alpha z_{r+1} \\
     & + (1+(\alpha+1) 2^\alpha z_{r+1}^\alpha) \partial_\alpha z_{r+1}''.
    \end{aligned}
  \end{equation*}
  Dividing the above by $z_r' = [1 + (\alpha+1) 2^\alpha z_{r+1}^\alpha] z_{r+1}'$
  and using Lemma \ref{lemma:Iansestimate} to bound $z_{r+1}$, 
  Lemma \ref{lemma:Izk} to bound $z_{r}'$, Lemma \ref{lemma:Ipzk} to bound
  $z_{r}''/z_{r}'$, Lemma \ref{lemma:dzalpha} to bound $\partial_\alpha z_r$
  and Lemma \ref{lemma:dIalpha} to bound $\partial_\alpha z_{r}' / z_r'$,
  we obtain for $r\geq 0$:
  \begin{align*}
    |z_r^{\alpha-1} z_r' \log z_r | 
    & \leq \const \, r^{-2} \, z_0^{-2} 
      \, ( \logg r - \log z_0 \, ), \\
    |z_r^{\alpha-2} z_r' \partial_\alpha z_r |
    & \leq \const \, r^{-2} \, z_0^{-2}
      \, \logg (r z_0^\alpha) \, ( \logg r- \log z_0 ), \\
    |z_r^{\alpha-1} \partial_\alpha z_r' |
    & \leq \const \, r^{-2} \, z_0^{-2}
      \, [\logg (r z_0^\alpha)]^2 \, ( \logg r - \log z_0 ), \\
    |z_r^\alpha z_r'' \log z_r / z_r' |
    & \leq \const \, r^{-2} \, z_0^{-2}
      \, ( \logg r - \log z_0 ), \\
    |z_r^{\alpha-1} z_r'' (\partial_\alpha z_r) / z_r' |
    & \leq \const \, r^{-2} \, z_0^{-2}
      \, \logg (r z_0^\alpha) \, ( \logg r - \log z_0 ).
  \end{align*}
  Hence for $r \geq 1$
  \[
    \left| \frac{\partial_\alpha z_{r}''}{z_{r}'}- \frac{\partial_\alpha z_{r-1}''}{z_{r-1}'} \right|
    \leq \const \, r^{-2} \, z_0^{-2}
      \, [\logg (r z_0^\alpha)]^2 \, ( \logg r - \log z_0 ).
  \]
  Recall that $\partial_\alpha z_0'' = 0$. Then
  \[
    \left|\frac{\partial_\alpha z_{r}''}{z_{r}'} \right|
    \leq z_0^{-2} \,\sum_{j=1}^r j^{-2}
    \, [\logg (j z_0^\alpha)]^2 \, \left[ \logg j - \log z_0 \right]
    \leq \const \, z_0^{-2} ( 1 - \log z_0 ).
  \]
\end{proof}

\subsection{Proof of Theorem \ref{th:lsv}}
  \label{subsection:proofofth:lsv}

The verification of assumptions of Theorem \ref{th:dhdalpha} is as follows.
Since $G_{\alpha,r}$ and $F_{\alpha,r}^{-1}$ are defined on $[1/2,1]$,
we use that $z = z_0 \geq 1/2$ in the bounds below. Now,
\begin{itemize}
  \item[\ref{A1}.] By equation (\ref{eq:hazzz}), $z_r'\leq 1$, thus $\|G_{\alpha,r}\|_\infty \leq 1/2$.
  \item[\ref{A2}.] By Lemma \ref{lemma:Ipzk}, $|z_r''/ z_r'| \leq \const$, thus
    $\|G_{\alpha,r}'/G_{\alpha,r}\|_\infty \leq \const$.
  \item[\ref{A3}.] By Lemma \ref{lemma:Ippzk}, $|z_r'''/ z_r' | \leq \const$, thus
    $\|G_{\alpha,r}''/G_{\alpha,r}\|_\infty \leq \const$.
  \item[\ref{A4}.] By Lemma \ref{lemma:dzalpha}, 
    $|\partial_\alpha z_r | \leq \const \, r^{-1/\alpha}\, (\logg r)^2 $, and by equation
    (\ref{eq:l93n}) we have
    \[
      \|\partial_\alpha F_{\alpha,r}^{-1}\|_\infty \leq \const \, r^{-1/\alpha}\, (\logg r)^2
      \leq \const \, (\logg r)^2.
    \]
  \item[\ref{A5}.] By Lemma \ref{lemma:dIalpha},
    $\left|(\partial_\alpha z_r')/z_r'\right| \leq \const \, (\logg r)^3$,
    thus $\|(\partial_\alpha G_{\alpha,r})/G_{\alpha,r}\|_\infty \leq \const \, (\logg r)^3$.
  \item[\ref{A6}.] By Lemma \ref{lemma:dIpalpha},
    $\left|(\partial_\alpha z_r'')/z_r'\right| \leq \const$,
    thus $\|(\partial_\alpha G_{\alpha,r}')/G_{\alpha,r}\|_\infty \leq \const$.
  \item[\ref{A7}.] By Remark \ref{remark:zzzzzz}, 
    $z_0'=1$, and by Lemma \ref{lemma:Izk}, $|z_r'| \leq \const \, r^{-(\alpha+1)/\alpha}$
    for $r\geq 1$, so
    \[
      \sum_{r=0}^\infty \| G_{\alpha,r} \|_\infty (\logg r)^3
      = \frac{1}{2} \sum_{r=0}^\infty \sup_z | z_r' | \cdot (\logg r)^3
      \leq \frac{1}{2} +  \const \, \sum_{r=1}^\infty \frac{(\logg r)^3}{r^{1+1/\alpha}} \leq \const.
    \]
\end{itemize}

To verify the assumptions of the Corollary \ref{cor:linresp} --- we have to show in addition that
\begin{itemize}
  \item $\sum_{r=0}^\infty (r+1) (\logg r)^3 \, \left\|G_{\alpha,r}\right\|_\infty \leq \const$.
    By Lemma \ref{lemma:Izk} and equation (\ref{eq:l93m}), 
    \[|G_{\alpha,r}(z)| = |z_r'|/2 \leq \const \, r^{-(\alpha+1)/\alpha},\]
    thus
    \[
      \sum_{r=0}^\infty (r+1) (\logg r)^3 \, \left\| G_{\alpha,r} \right\|_\infty 
      \leq \const \sum_{r=0}^\infty \frac{(\logg r)^3}{r^{1/\alpha}} \, \frac{r+1}{r}
      \leq \const.
    \]
  \item $\|\partial_\alpha [\iphi \circ F_{\alpha,r}^{-1}] \|_\infty 
    \leq  \const \, \|\varphi\|_{C^1} \, (r+1)$ and 
    $\|\iphi \circ F_{\alpha,r}^{-1}\|_\infty \leq  \const \, \|\varphi\|_{C^1} \, (r+1)$. 
    This is true because
      \[
        \left(\iphi \circ F_{\alpha,r}^{-1}\right)(z) 
        = \varphi \left(\frac{z_r+1}{2}\right) + \sum_{j=0}^{r-1} \varphi(T_\alpha^j z_r)
        = \varphi \left(\frac{z_r+1}{2}\right) + \sum_{j=1}^{r} \varphi(z_j),
      \]
      and $|\partial_\alpha z_r| \leq \const $ by Lemma \ref{lemma:dzalpha}.
\end{itemize}

Hence we have verified assumptions of Theorem \ref{th:dhdalpha}
and Corollary \ref{cor:linresp} as required.

\subsection{Proof of Theorem \ref{th:whatamIdoing}}
\label{subsection:proofofth:whatamIdoing}
  
  Recall that the invariant measure of $T_\alpha$ is denoted by $\nu_\alpha$,
  and its density by $\rho_\alpha$, while the invariant measure of the induced
  map $F_{\alpha}$ is denoted by $\mu_\alpha$, and its density by $h_\alpha$.
 
  \begin{lemma}
    \label{lemma:pmdenc}
    $\rho_\alpha(z) = g_\alpha(z) \big/ \int_0^1 g_\alpha(x)\, dx$ for all $z \in (0,1]$, where
    \[
      g_\alpha(z) = \frac{1}{2} \sum_{k=0}^\infty h_\alpha \left(\frac{z_k+1}{2}\right) \, z_k'.
    \]
  \end{lemma}
  \begin{proof}
    Let $\varphi$ be a nonnegative observable on $[0,1]$, and 
    $\Phi_\alpha = \sum_{k=0}^{\tau_\alpha -1} \varphi \circ T_\alpha^k$ 
    be the corresponding induced observable.
    In the beginning of Section~\ref{section:LSVintro} we partitioned the interval
    \([1/2,1]\) into intervals \([y_{r+1}, y_r]\), \(r \geq 0\),
    where \(F_\alpha\) has full branches and \(\tau_\alpha = r+1\).
    Compute
    \begin{align*}
      \int \Phi_\alpha \, d\mu_\alpha 
      & = \int_{1/2}^1 \sum_{k=0}^{\tau_\alpha(y)-1}
        \varphi(T_\alpha^k y ) \, h_\alpha(y) \, dy 
        = \sum_{j=0}^{\infty} \int_{y_{j+1}}^{y_j} \sum_{k=0}^{j} 
          \varphi(T_\alpha^k y) \, h_\alpha(y) \, dy  \\
      & = \sum_{k=0}^\infty \sum_{j=k}^\infty \int_{y_{j+1}}^{y_j}
           \varphi(T_\alpha^k y) \, h_\alpha(y) \, dy  
        = \sum_{k=0}^\infty \int_{1/2}^{y_k} 
           \varphi(T_\alpha^k y) \, h_\alpha(y) \, dy  \\
      & = \int_{1/2}^{1} \varphi(y) \, h_\alpha(y) \, dy 
        + \frac{1}{2} \sum_{k=1}^{\infty} \int_{0}^{x_k} 
           \varphi(T_\alpha^{k-1} x) \, h_\alpha\left(\frac{x+1}{2}\right) \, dx \\
      & = \int_{1/2}^{1} \varphi(y) \, h_\alpha(y) \, dy 
        + \frac{1}{2} \sum_{k=0}^{\infty} \int_{0}^{x_{k+1}} 
           \varphi(T_\alpha^{k} x) \, h_\alpha\left(\frac{x+1}{2}\right) \, dx \\
      & =  \int_{1/2}^{1} \varphi(y) \, h_\alpha(y) \, dy 
        + \frac{1}{2} \sum_{k=0}^{\infty}
           \int_{0}^{1/2} \varphi(z) \, h_\alpha \left(\frac{z_k+1}{2}\right) \, z_k' \, dz \\
      & = \frac{1}{2} \sum_{k=0}^{\infty}
           \int_{0}^{1} \varphi(z) \, h_\alpha \left(\frac{z_k+1}{2}\right) \, z_k' \, dz.
    \end{align*}
    First we made a substitution  $x= T_\alpha y = 2y-1$, and then a
    substitution $z = T_\alpha^{k}x$, i.e.\ $x = z_k$. In the last step we used the fact that
    for $z \geq 1/2$
    \[
      h_\alpha(z) = (P_\alpha h_\alpha)(z) 
      = \sum_{k=0}^\infty h_\alpha \left(\frac{z_k+1}{2}\right) \frac{z_k'}{2}.
    \]
    Since $\int \varphi \, d\nu_\alpha = \int \Phi_\alpha \, d\mu_\alpha \big/ 
    \int \tau_\alpha \, d\mu_\alpha $, the result follows.
  \end{proof}

  \begin{lemma}
    \label{lemma:pmdens}
    $g_\alpha(z)$ and $\partial_\alpha g_\alpha(z)$ are jointly continuous
    in $\alpha, z$ on \([\alpha_-, \alpha_+] \times (0, 1]\). Also,
    $0 \leq g_\alpha (z) \leq \const \, z^{-\alpha}$
    and
    $\left|\partial_\alpha g_\alpha(z) \right|\leq \const \, z^{-\alpha} (1-\log z)^3$.
  \end{lemma}
  \begin{proof}
    By Theorem \ref{th:dhdalpha}, $\|h_\alpha\|_{C^2} \leq \const$ and 
    $\|\partial_\alpha h_\alpha\|_{C^1} \leq \const$, and both 
    $h_\alpha(z)$ and $\partial_\alpha h_\alpha(z)$ are jointly continuous
    in $\alpha$ and $z$.
    By Lemma \ref{lemma:Izk}, 
    $
      0 \leq z_r' \leq \const \, 
      \left(1+r z^\alpha \alpha 2^\alpha\right)^{-(\alpha+1)/\alpha}
    $, hence
    \begin{equation}
      \label{eq:l54n8s}
      0 \leq\sum_{r=1}^{\infty} z_r' \leq \const \int_{1}^{\infty}
      \left(1+t z^\alpha \alpha 2^\alpha\right)^{-(\alpha+1)/\alpha} \, dt
      \leq \const z^{-\alpha}.
    \end{equation}
    Now,
    \[
      0 \leq g_\alpha(z) =\frac{1}{2} \sum_{k=0}^{\infty} h_\alpha \left(\frac{z_k+1}{2}\right) \, z_k' 
      \leq \const \, z^{-\alpha}.
    \]
    Terms of the series are jointly continuous in $\alpha$ and $z$, and convergence is uniform
    away from $z=0$, thus $g_\alpha(z)$ is also jointly continuous in $\alpha$ and $z$.
    
    Denote $u_{\alpha,k}(z) = h_\alpha \left((z_k+1)/2\right) \, z_k' / 2$,
    so that $g_\alpha(z) = \sum_{k=0}^\infty u_{\alpha,k}(z)$ and compute
    \[
      \partial_\alpha u_{\alpha,k}(z)
      = \left[(\partial_\alpha h_\alpha) \left(\frac{z_k+1}{2}\right)
          + h_\alpha \left(\frac{z_k+1}{2}\right)
        \frac{\partial_\alpha z_k}{2}
        \right] \frac{z_k'}{2}
        + h_\alpha \left(\frac{z_k+1}{2}\right) \frac{\partial_\alpha z_k'}{2}.
    \]
    By Lemma \ref{lemma:dzalpha}, 
    \[
      0 \leq \partial_\alpha z_r \leq \const \,
      r^{-1/\alpha} \logg (r z^\alpha) \, \left[\logg r - \log z \right].
    \]
    By Lemma \ref{lemma:dIalpha},
    \[
      |\partial_\alpha z_r'| \leq \const \, z_r' \,
      [\logg (r z^\alpha)]^2 \, \left[\logg r -\log z \right].
    \]
    Thus $|\partial_\alpha u_{\alpha,k}(z)| \leq \const \, z_r' \, 
    [\logg (r z^\alpha)]^2 \, \left[\logg r -\log z \right]$.
    Thus by Lemma \ref{lemma:Izk},
    \begin{align*}
      \sum_{k=0}^{\infty} |\partial_\alpha u_{\alpha,k}(z)|
      &\leq \const \,\sum_{k=0}^{\infty} z_k' \, 
        [\logg (k z^\alpha)]^2 \,\left[\logg k -\log z \right] \\
      &\leq \const \, \int_{1}^{\infty} 
        \left(1+t z^\alpha \alpha 2^\alpha\right)^{-(\alpha+1)/\alpha}\,
        [\log (t z^\alpha)]^2 \,\left[\log t -\log z \right] \, dt \\
      & = \const \, z^{-\alpha} \, \int_{z^\alpha}^{\infty} 
        \left(1 + s \alpha 2^\alpha\right)^{-(\alpha+1)/\alpha}\,
        (\log s)^2 \,\left[\log \frac{s}{z^{\alpha}} -\log z \right] \, ds \\
      & \leq \const \, z^{-\alpha} \, (1-\log z).
    \end{align*}
    Therefore we can write
    \[
      (\partial_\alpha g_\alpha)(z) 
      = \sum_{k=0}^{\infty} \partial_\alpha u_{\alpha,k}(z).
    \]
    Away from $z=0$, the terms of the series are jointly continuous in 
    $\alpha$ and $z$, and series converges uniformly, so
    $(\partial_\alpha g_\alpha)(z)$ is jointly continuous in $\alpha$ and $z$,
    and
    $
      \left|\partial_\alpha g_\alpha(z) \right|
      \leq \const \, z^{-\alpha} (1-\log z).
    $
  \end{proof}
  
  \begin{corollary}
    \label{lemma:pmden0}
    $\rho_\alpha(z)$ and $\partial_\alpha \rho_\alpha(z)$ are jointly continuous
    in $\alpha$ and $z$. Also,
    $0 \leq g_\alpha (z) \leq \const \, z^{-\alpha}$
    and
    $\left|\partial_\alpha g_\alpha(z) \right|\leq \const \, z^{-\alpha} (1-\log z)$.    
  \end{corollary}
  \begin{proof}
    Note that $\int_{0}^{1} \, g_\alpha (z) \, dz = \int \tau_\alpha \, d\mu_\alpha \geq 1$, 
    and
    \[
      \frac{d}{d\alpha} \int _{0}^{1} \, g_\alpha (z) \, dz 
      = \int _{0}^{1} \, (\partial_\alpha g_\alpha) (z) \, dz.
    \]
    By Lemma \ref{lemma:pmdens}, $\int_{0}^{1} \, g_\alpha (z) \, dz$ is 
    continuously differentiable in $\alpha$, its derivative is bounded by
    $\const$. The result follows from Lemma \ref{lemma:pmdens} and relation,
    established in Lemma \ref{lemma:pmdenc}:
    \[
      \rho_\alpha(z) = g_\alpha(z) \Big/ \int_0^1 g_\alpha(x) \, dx.
    \]
  \end{proof}
  
  \begin{corollary}
    Assume that $\varphi \in L^q[0,1]$, where $q > (1-\alpha_+)^{-1}$. Then the map
    $\alpha \mapsto \int \varphi(x) \rho_\alpha(x) \, dx$ is continuously
    differentiable on $[\alpha_-, \alpha_+]$.
  \end{corollary}
  \begin{proof}
    Let $p = 1/(1-1/q)$. Then $p < 1/\alpha_+$ and by Corollary \ref{lemma:pmden0},
    $\|\partial_\alpha \rho_\alpha\|_{L^p}$
    is bounded uniformly in $\alpha$.
    Since $\rho_\alpha(x)$ and $(\partial_\alpha \rho_\alpha) (x)$ are jointly
    continuous in $\alpha$ and $x$, we can write
    \[
      \left|\frac{d}{d\alpha} \int \varphi(x) \rho_\alpha(x) \, dx \right|
      = \left|\int_0^1 \varphi(x) \, (\partial_\alpha \rho_\alpha)(x) \, dx \right|
      \leq \|\varphi\|_{L^q} \,
      \|\partial_\alpha \rho_\alpha\|_{L^p}.
    \]
    It is clear that the above is bounded on $[\alpha_-, \alpha_+]$. Continuity
    of $\int_0^1 \varphi(x) \, (\partial_\alpha \rho_\alpha) (x) \, dx$ follows
    from continuity of $(\partial_\alpha \rho_\alpha) (x)$ in $\alpha$ and the dominated
    convergence theorem.
  \end{proof}

\paragraph{Acknowledgements.}
This research was supported in part by a European Advanced Grant {\em StochExtHomog} (ERC AdG 320977).
The author is grateful to Zemer Kosloff and Ian Melbourne for many hours of discussions, 
and numerous suggestions and corrections on the manuscript.


\begin{thebibliography}{LSV99}

\bibitem[BS15]{BS15} W. Bahsoun and B. Saussol,
\emph{Linear response in the intermittent family: differentiation in a weighted 
\mbox{\(C^0\)-norm},}
arXiv:1512.01080 [math.DS].

\bibitem[B07]{B07} V. Baladi,
\emph{On the susceptibility function of piecewise expanding interval maps,} 
Comm. Math. Phys.
\textbf{275}  (2007),
839--859.

\bibitem[BS08]{BS08} V. Baladi and D. Smania, 
\emph{Linear response formula for piecewise expanding unimodal maps,} 
Nonlinearity
\textbf{21} (2008),
677--711.
(Corrigendum:   Nonlinearity \textbf{25} (2012), 
2203--2205.)

\bibitem[BT15]{BT15} V. Baladi and M. Todd, 
\emph{Linear response for intermittent maps,}
to appear in Comm. Math. Phys. \\
arXiv:1508.02700 [math.DS]

\bibitem[D04]{D04} D. Dolgopyat,
\emph{On differentiability of SRB states for partially hyperbolic systems,}
Invent. Math.
{\bf 155}  (2004),
389--449.

\bibitem[G04]{G04} S. Gou\"ezel, 
\emph{Sharp polynomial estimates for the decay of correlations,}
Israel J. Math.,
{\bf 139} (2004),
29--65.

\bibitem[H04]{H04} H. Hu,
\emph{Decay of correlations for piecewise smooth maps with indifferent fixed points,} 
Ergodic Theory Dynam. Systems
\textbf{24} (2004), 
495--524.

\bibitem[LSV99]{LSV99} C. Liverani, B. Saussol, and S. Vaienti,
\emph{A probabilistic approach to intermittency,}
Ergodic Theory Dynam. Systems
{\bf 19} (1999),
671--685.

\bibitem[M07]{M07} M. Mazzolena, 
\emph{Dinamiche espansive unidimensionali: dipendenza della misura invariante da un parametro,}
Master's Thesis, Roma 2 (2007).

\bibitem[P80]{P80} G. Piangiani,
\emph{First return map and invariant measures,}
Israel J. Math.
\textbf{35} (1980),
32--48.

\bibitem[R97]{R97} D. Ruelle, 
\emph{Differentiation of SRB states,}
Comm. Math. Phys.
\textbf{187} (1997),
227--241.

\bibitem[R98]{R98} D. Ruelle,  
\emph{General linear response formula in statistical mechanics, and the 
fluctuation-dissipation theorem far from equilibrium,} 
Phys. Lett. A
\textbf{245} (1998),
220--224.

\bibitem[R09]{R09} D. Ruelle, 
\emph{Structure and $f$-dependence of the {A.C.I.M.} for a unimodal map $f$ of 
{M}isiurewicz type,}
Comm. Math. Phys.
\textbf{287} (2009),
1039--1070.
   
\bibitem[R09.1]{R09.1} D. Ruelle,  
\emph{A review of linear response theory for general differentiable dynamical systems,} 
Nonlinearity
\textbf{22} (2009), 
855--870.

\bibitem[S02]{S02} O. Sarig, 
\emph{Subexponential decay of correlations,}
Invent. Math.
\textbf{150} (2002),
629--653. 

\bibitem[Y99]{Y99} L.-S. Young,
\emph{Recurrence times and rates of mixing,}
Israel J. Math.
\textbf{110} (1999), 
153--188.

\bibitem[Z04]{Z04} R.~Zweim{\"u}ller,
\emph{Kuzmin, coupling, cones, and exponential mixing,} 
Forum Math. 
\textbf{16} (2004),
447--457.
 
\end{thebibliography}
\end{document}